\documentclass[11pt]{article}
\date{}

\usepackage[title]{appendix}
\usepackage{xcolor}
\usepackage[margin=1in]{geometry}                
\usepackage{graphicx}
\usepackage{subcaption}
\usepackage{pdflscape}
\usepackage{amssymb}
\usepackage[normalem]{ulem}
\usepackage{hyperref}
\usepackage{enumitem}
\usepackage{mathrsfs}
\usepackage{epstopdf}
\usepackage{rotating}
\usepackage{longtable} 
\usepackage{adjustbox}
\usepackage{float}

\usepackage{color}
\usepackage{bbm, dsfont}
\usepackage{pst-node}
\usepackage{tikz-cd}
\usepackage{amsfonts} 
\usepackage{geometry}
\usepackage{amsthm}
\usepackage{amsmath}
\usepackage{titlesec}
\usepackage{amssymb}
\usepackage{enumitem}
\usepackage{float}
\usepackage [english]{babel}
\usepackage [autostyle, english = american]{csquotes}
\usepackage{algorithm}
\usepackage[noend]{algpseudocode} 
\makeatletter
\def\BState{\State\hskip-\ALG@thistlm}
\makeatother
\usepackage{hyperref}
\usepackage[normalem]{ulem}
\usepackage{mathrsfs}
\usepackage[italicdiff]{physics}

\newlist{casess}{enumerate}{1}
\setlist[casess]{label=     \textbf{Case} \arabic*:}
\usepackage{mathtools}

\makeatletter
\newcommand*{\rom}[1]{\expandafter\@slowromancap\romannumeral #1@}
\makeatother

\usepackage{etoolbox}

\makeatletter
\patchcmd{\ttlh@hang}{\parindent\z@}{\parindent\z@\leavevmode}{}{}
\patchcmd{\ttlh@hang}{\noindent}{}{}{}
\makeatother

\usepackage{listings}
\usepackage{color} 
\definecolor{mygreen}{RGB}{28,172,0} 
\definecolor{mylilas}{RGB}{170,55,241}

\newlist{Assumptions}{enumerate}{1}
\setlist[Assumptions]{label=     \textbf{Assumption} \arabic*:}

\makeatletter

\newsavebox{\@brx}
\newcommand{\llangle}[1][]{\savebox{\@brx}{\(\m@th{#1\langle}\)}%
  \mathopen{\copy\@brx\kern-0.5\wd\@brx\usebox{\@brx}}}
\newcommand{\rrangle}[1][]{\savebox{\@brx}{\(\m@th{#1\rangle}\)}%
  \mathclose{\copy\@brx\kern-0.5\wd\@brx\usebox{\@brx}}}
\makeatother

\usepackage{lipsum} 
\usepackage{titlesec}
\titleformat{\subsection}[runin]
       {\normalfont\bfseries}
       {\thesubsection}
       {0.5em}
       {}
       [.]

 \newtheorem{thm}{Theorem}[section]
 \newtheorem{cor}[thm]{Corollary}
 
 \newtheorem{lem}[thm]{Lemma}
 \newtheorem{prop}[thm]{Proposition}
  \newtheorem{exl}[thm]{Example}
 \theoremstyle{definition}
 \newtheorem{defn}[thm]{Definition}
 \theoremstyle{remark}

 \numberwithin{equation}{section}

\numberwithin{equation}{section}

\newcommand{\C}{\mathbb{C}}

\newcommand{\Hi}{\mathcal{H}}
\newcommand{\Ki}{\mathcal{K}}
\newcommand{\acts}{\curvearrowright}
\newcommand{\N}{\mathbb{N}}
\newcommand{\id}{\textup{id}}

\newcommand{\Ind}{\textup{Ind}}

\def\N{\mathbb{N}}
\def\T{\mathbb{T}}
\def\Z{\mathbb{Z}}

\def\R{\mathbb{R}}
\def\Z{\mathbb Z}

\newcommand{\Aut}{{\rm Aut}}

\DeclarePairedDelimiterX{\inp}[2]{\langle}{\rangle}{#1, #2}

\makeatletter
\newcommand*\bigcdot{\mathpalette\bigcdot@{.5}}
\newcommand*\bigcdot@[2]{\mathbin{\vcenter{\hbox{\scalebox{#2}{$\m@th#1\bullet$}}}}}
\makeatother

\def\CC{\mathbb C}

\def\<{\langle}
\def\>{\rangle}

\numberwithin{equation}{section}

\usepackage[backend=biber,maxnames=10]{biblatex}
\begin{document}

\title{Simplicity and pure infiniteness for $\textup{C}^*$-algebras associated with stabilizers of boundary actions}

\author{Felipe Flores \& Joseph Gondek
\footnote{
\newline
For the purpose of open access, the authors have applied a CC-BY license to any author accepted manuscript arising from this submission.}
}

\maketitle

\begin{abstract}\setlength{\parindent}{0pt}\setlength{\parskip}{1ex}\noindent
Given a discrete group $\Gamma$ acting on a compact space $X$, and a stabilizer subgroup $\Lambda\leq \Gamma$ of $X$, we use the Rieffel induction of covariant  $(\Lambda, X)$-representations to study classes of representations induced by certain characters on $\Lambda$ and the $\textup{C}^*$-algebras they generate. When $X$ is a $\Gamma$-boundary, we obtain new classes of simple, traceless group $\textup{C}^*$-algebras. When the $\Gamma$-boundary $X$ is an extreme boundary, we show that the associated group $\textup{C}^*$-algebras are also purely infinite, answering part of a question posed by Kalantar and Scarparo on $\textup{C}^*$-algebras associated with Thompson's groups. The latter $\textup{C}^*$-algebras are shown to be selfless in the sense of Robert.

\end{abstract}

\section{Introduction}

The past decade of research in operator algebras and analytic group theory has revealed several dynamical properties of a discrete group $\Gamma$ that are reflected in its associated group $\textup{C}^*$-algebras. In \cite{KK14}, Kalantar and Kennedy characterize the simplicity of the reduced $\textup{C}^*$-algebra of $\Gamma$ by the existence of a topologically free $\Gamma$-boundary; more recently, Ozawa found that the existence of a topologically free extreme $\Gamma$-boundary is a sufficient condition for the selflessness (in the sense of Robert \cite{Ro25}) of the reduced $\textup{C}^*$-algebra of $\Gamma$ with respect to its canonical tracial state \cite{Oz25}. Consequently, several desirable properties, such as strict comparison and stable rank one, have been established for large classes of reduced group $\textup{C}^*$-algebras \cite{AGKEP25,BaFl26,,FKOCP26,Oz25,RTV25,Vi26}, including the reduced $\textup{C}^*$-algebra of the free group $\mathbb F_n$.

In this article, we consider selflessness, simplicity, and traces for $\textup{C}^*$-algebras generated by other unitary representations of $\Gamma$. Specifically, we consider representations of $\Gamma$ induced from characters defined on the stabilizer of a point in a $\Gamma$-boundary. In the case where the boundaries are topologically free, these stabilizers become trivial, and we recover some of the results in \cite{BKKO17} and \cite{Oz25}. Our results follow the refinements of Kawabe \cite{Ka17} and Kalantar-Scarparo \cite{KS2022} on the $\textup{C}^*$-simplicity problem, which show, for example, that the $\textup{C}^*$-algebra generated by the quasi-regular representation of a stabilizer from a $\Gamma$-boundary contains a unique maximal ideal.

In the first half of this paper, we study the ideal and trace structure of these $\textup{C}^*$-algebras for \textit{relatively central} characters $\chi$ (see Definition \ref{rc}). The main feature of these representations is that their associated $\textup{C}^*$-algebras admit unique $\Gamma$-equivariant ucp maps into $C(\partial_F\Gamma)$, where $\partial_F\Gamma$ is the \textit{Furstenberg boundary} of $\Gamma$.

\begin{thm}[see Corollary \ref{uniqueness}]\label{unique}
Let $X$ be a $\Gamma$-boundary, $x\in X$, and $\Lambda\leq\Gamma$ a subgroup with $\Gamma_x^\circ\leq \Lambda\leq\Gamma_x$. If $\chi: \Lambda\to \C$ is a relatively central character, then there is a unique $\Gamma$-equivariant ucp map $\varphi:\textup{C}^*_{\Ind_{\Lambda}^\Gamma\chi}\Gamma\to C(\partial_F\Gamma)$. 
\end{thm}

These maps, called \textit{boundary maps}, have previously been effective in the study of the reduced $\textup{C}^*$-algebra of $\Gamma$, as well as $\textup{C}^*$-algebras associated with \textit{germinal} representations of $\Gamma$ (see \cite{KS2022, Ke20}). Using Theorem \ref{unique}, we recover several analogues of the main results from \cite{KS2022}. For example, the next result is the analogue of \cite[Proposition 4.8]{KS2022}.

\begin{cor}[see Theorem \ref{traces}]
Suppose $X$ is a faithful $\Gamma$-boundary. For a subgroup $\Lambda\leq\Gamma$ with $\Gamma_x^\circ\leq \Lambda\leq\Gamma_x$ and a relatively central character $\chi: \Lambda\to \C$, the \textup{C}$^*$-algebra $\textup{C}^*_{\Ind_{\Lambda}^\Gamma\chi}\Gamma$ admits a trace if and only if $X$ is topologically free. Furthermore, if $X$ is topologically free, then $\lambda_\Gamma\prec \Ind_{\Lambda}^\Gamma\chi$, and the canonical trace is the unique trace on $\textup{C}^*_{\Ind_{\Lambda}^\Gamma\chi}\Gamma$.
\end{cor}

Following the notation from \cite{LBMB18}, we let $X_0\subset X$ denote the locus of continuity points for the stabilizer map $\textup{Stab}: X\ni x\mapsto \Gamma_x\in \text{Sub}(\Gamma)$, where the codomain is equipped with the Chabauty topology. The next result is the analogue of \cite[Theorem 5.2]{KS2022}.

\begin{thm}[see Proposition \ref{simple}]
The \textup{C}$^*$-algebra $\textup{C}^*_{\Ind_{\Gamma_x^\circ}^\Gamma\chi}\Gamma$ is simple for every $x\in X_0$ and every relatively central character $\chi: \Gamma_x^\circ\to \C$.
\end{thm}

A construction central to all the results in this article, which appears to be new, is a model for the \textit{Rieffel induction} of a covariant $(\Lambda, X)$-representation, where $X$ is a compact $\Gamma$-space and $\Lambda\leq\Gamma$. Recall that there is a canonical conditional expectation $\mathbb{E}_H: C(X)\rtimes \Gamma\to C(X)\rtimes \Lambda$. Integrating a covariant $(\Lambda, X)$-representation on a Hilbert space $\mathbb{\Hi}$ to a $*$-representation $\Phi:C(X)\rtimes \Lambda\to \mathbb{B}(\Hi)$, Rieffel induction with respect to $\mathbb{E}_H$ produces a $*$-representation $\Ind_{\Lambda}^\Gamma\Phi$ of $C(X)\rtimes \Gamma$, and thus a covariant $(\Gamma, X)$-representation which we call the \textit{induced} covariant representation. Our elementary observation is that the representation $\Ind_{\Lambda}^\Gamma\Phi$ can be concretely modeled on $\Ind_{\Lambda}^\Gamma\Hi$. As a consequence, we construct a new class of covariant representations (Proposition \ref{newgerminal}) that are germinal in the sense of Kalantar and Scarparo.

The second half of this article concerns pure infiniteness for these group $\textup{C}^*$-algebras. Our interest in this property comes from the main result of \cite{Ki95,KiPh00,Ph00}, which shows that $KK$-theory is a complete invariant for simple, separable, nuclear, purely infinite $\textup{C}^*$-algebras which satisfy the UCT (\textit{Kirchberg} algebras). It is therefore desirable, whenever possible, to identify group $\textup{C}^*$-algebras which are Kirchberg algebras. Our main result in this section resolves pure infiniteness for certain classes of simple group $\textup{C}^*$-algebras.

Our approach uses the Powers-Haagerup-Pisier (PHP) property defined in \cite{Oz25} (see Section \ref{sect4}), which implies the \textit{selflessness} of the $\textup{C}^*$-algebra with respect to its canonical vector state (\cite[Theorem 3.3]{HKEPR25}). While we will not explicitly use selflessness in this paper, we record the definition here for posterity.

\begin{defn}[\cite{Ro25}]
Let $\mathcal A$ be a unital $\textup{C}^*$-algebra equipped with a state $\rho$ that induces a faithful GNS representation. The pair $(\mathcal A,\rho)$, referred to as a \emph{$\textup{C}^*$-probability space}, is said to be \emph{selfless} if $\mathcal A\not\cong\mathbb{C}$ and the first factor embedding into the reduced free product $\textup{C}^*$-algebra $\iota\colon (\mathcal A,\rho) \hookrightarrow (\mathcal A,\rho) \star (\mathcal A,\rho)$ is \emph{existential}, i.e., there exists an ultrafilter $\mathcal{U}$ and an embedding $\theta\colon (\mathcal A,\rho)\star (\mathcal A,\rho) \hookrightarrow (\mathcal A^{\mathcal{U}},\rho^{\mathcal{U}})$ such that the diagram
\vspace{-3mm}
\[\begin{tikzcd}
	{(\mathcal A,\rho)} && {(\mathcal A^{\mathcal{U}},\rho^{\mathcal{U}})} \\
	& {(\mathcal A,\rho)\star(\mathcal A,\rho)}
	\arrow["\delta", hook, from=1-1, to=1-3]
	\arrow["{\iota}", hook, from=1-1, to=2-2]
	\arrow["\theta", hook, from=2-2, to=1-3]
\end{tikzcd}\]
\vspace{-2mm}
commutes, where $\delta$ denotes the diagonal embedding.
\end{defn}

Examples of discrete groups that generate selfless reduced $\textup{C}^*$-algebras include acylindrically hyperbolic groups with trivial finite radical \cite{Oz25} and linear groups with trivial amenable radical \cite{Vi26}. Selflessness has been established for large classes of free products \cite{FKOCP25,HKEPR25, HKER25}, amalgamated free products \cite{GKEPL26}, and graph products \cite{FKOCP25}; this property has also seen recent generalizations to the relative setting \cite{GJKEPR26} and the setting of inclusions \cite{HKEPR25}.

The relevance of selflessness in this article comes from a standard theorem of Robert \cite[Theorem 3.1]{Ro25}, which states that a selfless, traceless $\textup{C}^*$-algebra is automatically purely infinite. In \cite{Oz25}, Ozawa proves that if $\Gamma$ admits a topologically free extreme boundary action, then $\textup{C}^*_r\Gamma$ has Property PHP (and is therefore selfless) with respect to its canonical vector state. Our next theorem can be viewed as an extension of this result to the non-topologically free setting.

\begin{thm}[see Theorem \ref{selfthm}]\label{idkmain}
    Let $X$ be an extreme $\Gamma$-boundary with $\#X>2$. For every $x\in X_0$ and every centrally extendable character $\chi:\Gamma_x\to\T$, the $\textup{C}^*$-algebra $\textup{C}^*_{\Ind_{\Gamma_x}^\Gamma\chi}\Gamma$ has Property \textup{PHP} with respect to the canonical vector state $\varphi$. In particular, the pair $(\textup{C}^*_{\Ind_{\Gamma_x}^\Gamma\chi}\Gamma,\varphi)$ is selfless and we have the following dichotomy:
    \begin{enumerate}
        \item[(i)] If $\Gamma\curvearrowright X$ is topologically free, then $\textup{C}^*_{\Ind_{\Gamma_x}^\Gamma\chi}\Gamma=\textup{C}^*_{r}\Gamma$ is simple, monotracial, has stable rank one, and strict comparison.
        \item[(ii)] If $\Gamma\curvearrowright X$ is not topologically free, then $\textup{C}^*_{\Ind_{\Gamma_x}^\Gamma\chi}\Gamma$ is simple and purely infinite.
    \end{enumerate}
\end{thm}

We finish this article with an application to the \textit{Thompson-Stein} groups. These are groups of piecewise linear homeomorphisms introduced by Stein \cite{St92} as a generalization of the Thompson and Higman-Thompson groups \cite{Hi74}. The construction of these groups requires a triple $(\Gamma,\Lambda,\ell)$, where $\Lambda\leq \R^+$ is a nontrivial subgroup, $\Gamma$ is a $\Z\Lambda$-submodule of $\R$ such that $\Lambda\Gamma=\Gamma$, and $1\leq \ell\in\Gamma$. With these parameters, Stein constructed groups $F(\Gamma,\Lambda,\ell)\leq T(\Gamma,\Lambda,\ell)\leq V(\Gamma,\Lambda,\ell)$, generalizing Thompson's $F\leq T\leq V$. For the precise definitions, we refer to Section \ref{application}.

\begin{thm}[see Corollary \ref{TScor}]\label{TSTHM}
    Set $F=F(\Gamma,\Lambda,\ell)$ and $T=T(\Gamma,\Lambda,\ell)$. The $\textup{C}^*$-algebra $\textup{C}^*_{\lambda_{T/F}}T$ is simple and purely infinite.
\end{thm}

A technicality in the proof of Theorem \ref{TSTHM} is that $F$ occurs as a stabilizer $T_x$ of a non-generic point $x\in X\setminus X_0$ for the canonical extreme boundary action $T\acts X = \R/\ell\Z$, so that Theorem \ref{idkmain} cannot be directly applied. This is resolved with the observation that $T_x/T_x^\circ$ is an abelian and thus amenable group; see Lemma \ref{constant} and Theorem \ref{TSstuff}.

The result in Theorem \ref{TSTHM} makes progress on an operator-algebraic problem related to the Thompson groups $F\leq T\leq V$. In \cite{KS2022}, Kalantar and Scarparo show that $\textup{C}^*_{\lambda_{T/F}}T$ is simple, traceless, and not stably finite, and therefore a candidate for a Kirchberg algebra. Theorem \ref{TSTHM} confirms that $\textup{C}^*_{\lambda_{T/F}}T$ is simple, separable, and purely infinite. The nuclearity of $\textup{C}^*_{\lambda_{T/F}}T$ currently remains open.

\section{Preliminaries}\label{prelim}

\subsection*{Topological dynamics} Throughout this paper, $\Gamma$ will denote a countable discrete group with neutral element $e$. The conjugacy class of an element $g\in\Gamma$ will be denoted by $\textup{Cl}(g)$. Given a compact space $X$, we will write $\Gamma\acts X$ for an action of $\Gamma$ by homeomorphisms on $X$. For $g\in\Gamma$ and $x\in X$, we set
\begin{itemize}
\item $\text{Fix}(g) = \{x\in X: g x = x\}$;
\item $\Gamma_x = \{g\in\Gamma: g x = x\}$; and
\item $\Gamma_x^{\circ} = \{g\in\Gamma: g\text{ fixes pointwise a neighborhood of $x$}\}$.
\end{itemize}

We say that the action is 
\begin{itemize}
\item \textit{minimal} if $ X$ has no proper, non-empty, closed $\Gamma$-invariant subspaces;
\item \textit{topologically free} if, for all $e\neq g\in\Gamma$, $\text{Fix}(g)$ has empty interior;
\item \textit{strongly proximal} if, for every $\nu\in \text{Prob}(\Gamma)$, its $\Gamma$-orbit satisfies $\overline{\Gamma\nu}\cap X\neq\emptyset$;
\item \textit{extremely proximal} if, for every pair of nonempty open subsets $U, V\subset  X$, there is $g\in\Gamma$ with $g(X\setminus U)\subset V$.
\end{itemize}
The $\Gamma$-space $X$ is called a \textit{$\Gamma$-boundary} if it is minimal and strongly proximal, and an \textit{extreme $\Gamma$-boundary} if it is extremely proximal. We note that every extreme boundary is minimal, and when $\#X>2$, it is also a boundary \cite[Theorem 2.3]{Gl74}.

We recall that there is a $\Gamma$-boundary $\partial_F\Gamma$, characterized up to $\Gamma$-equivariant homeomorphism by the universal property that every $\Gamma$-boundary is a $\Gamma$-equivariant factor of $\partial_F\Gamma$. The $\Gamma$-space $\partial_F\Gamma$ is called the \textit{Furstenberg boundary}. Given any $\Gamma$-boundary $X$, we will denote by $\mathfrak{b}_X: \partial_F\Gamma\to X$ the unique $\Gamma$-map onto $X$. For each $g\in\Gamma$, we set $\Delta_g = \overline{\mathfrak{b}_X^{-1}(\text{Fix}(g)^\circ)}$.

Equipped with the product topology, the set $\text{Sub}(\Gamma)$ of subgroups of $\Gamma$ is a compact space, called the \textit{Chabauty space} of $\Gamma$. We denote by $X_0$ the set of continuity points for the stabilizer map $X\ni x\mapsto \Gamma_x\in \text{Sub}(\Gamma)$. We recall that $X_0$ is a dense $G_\delta$ subset of $X$, and that $x\in X_0$ if and only if $\Gamma_x = \Gamma_x^\circ$ \cite[Lemma 2.2, Proposition 2.4]{LBMB18}.

\subsection*{\textup{C}$^*$-dynamical systems} Let $\mathcal{A}$ and $\mathcal{B}$ be unital $\textup{C}^*$-algebras. A linear map $\varphi: \mathcal{A}\to \mathcal{B}$ is called \textit{completely positive} if the amplifications $\varphi\otimes \id_n: \mathcal{A}\otimes \mathbb{M}_n\C\to \mathcal{B}\otimes \mathbb{M}_n\C$ are positive for every $n\in\mathbb{N}$. If $\varphi$ is unital, then $\varphi$ is called a \textit{unital completely positive} (ucp) map. We recall that by the Schwarz inequality \cite[Proposition 3.3]{Paulsen2003}, the faithful kernel $I_\varphi := \{a\in \mathcal{A}: \varphi(a^*a) = 0\}$ is a left ideal in $\mathcal{A}$ for any ucp $\varphi: \mathcal{A}\to \mathcal{B}$. 

Throughout this paper, an action of a group $\Gamma$ on a unital $\textup{C}^*$-algebra will be understood to be by $*$-automorphisms.

\begin{thm}[\cite{KK14}, Theorem 3.12]\label{rigidity}
The $\Gamma$-\textup{C}$^*$-algebra $C(\partial_F\Gamma)$ is injective in the category of $\Gamma$-\textup{C}$^*$-algebras with $\Gamma$-ucp maps. Furthermore, if $X$ is a $\Gamma$-boundary, then the unique $\Gamma$-equivariant ucp $\varphi: C( X)\to C(\partial_F\Gamma)$ is given by $\varphi(f) = f\circ \mathfrak{b}_X$ \textup{(}$f\in C( X)$\textup{)}.
\end{thm}

Let $\mathcal{A}$ be a unital $\Gamma$-$\textup{C}^*$-algebra. A $\Gamma$-ucp map $\varphi: \mathcal{A}\to C(\partial_F\Gamma)$ is called a \textit{boundary map}. The nonempty convex space of all boundary maps on $\mathcal{A}$ is denoted by $\textup{BD}(\mathcal{A})$.

\subsection*{Unitary representations and induction} 
Let $X$ be a compact $\Gamma$-space, and $\nu$ a $\sigma$-finite, quasi-invariant measure on $X$. The \textit{Koopman representation} of the triple $(\Gamma, X, \nu)$ is the covariant representation $(L^2(X, \nu), \kappa, \rho)$ defined by $[\kappa(g)\xi](x) =_{\text{a.e.}} (\frac{\textup{d}g\nu}{\textup{d}\nu})^{1/2}\xi(g^{-1}x)$ and $[\rho(f)\xi](x) =_{\text{a.e.}}f(x)\xi(x)$ ($x\in X$, $g\in\Gamma$, $f\in C(X)$.)

We now recall the \textit{induction} of unitary representations. Let $\Lambda$ be a subgroup of $\Gamma$, and $(\sigma, \Ki)$ a unitary representation of $\Lambda$. We will consider, up to unitary equivalence, the following three equivalent models for the resulting \textit{induced representation} $(\Ind_{\Lambda}^\Gamma\sigma, \Ind_{\Lambda}^\Gamma\Ki)$ of $\Gamma$.
\begin{enumerate}
\item[(1)]\label{induction1} Fix a left transversal $\mathcal{R}$ of $\Gamma/\Lambda$, and consider the Hilbert space $\tilde{\Ki}$ of all functions $\xi: \Gamma\to \mathcal{\Ki}$ for which (a) $\xi(th) = \sigma(h^{-1})\xi(t)$ for all $t\in\Gamma$, $h\in \Lambda$ and (b) $\sum_{r\in\mathcal{R}}\|f(r)\|^2 < \infty$. With pointwise operations and the inner product defined by the formula $\langle \xi, \eta\rangle = \sum_{r\in\mathcal{R}}\langle \xi(r), \eta(r)\rangle$, $\tilde{\Ki}$ is a well-defined Hilbert space. Note that the series in (b) and the formula for the inner product do not depend on the choice of left transversal $\mathcal{R}$. The induced representation of $(\sigma, \Ki)$ can be defined as the unitary representation $(\tilde{\sigma}, \tilde{\Ki})$ given by
\begin{align*}
[\tilde{\sigma}(g)\xi](t) = \xi(g^{-1}t),
\end{align*}
for $\xi\in\tilde{\Ki}$ and $g, t\in\Gamma$.
\item[(2)]\label{induction2} (Cocycle formula) Fix a left cocycle $c: \Gamma\times \Gamma/\Lambda\to \Lambda$, and consider the unitary representation $(\sigma', \ell^2(\Gamma/\Lambda, \mathcal{K}))$ given by
\begin{align*}
[\sigma'(g)\xi](t\Lambda) = \sigma(c(g^{-1}, t\Lambda)^{-1})\xi(g^{-1}t\Lambda),
\end{align*}
for $g, t\in\Gamma$ and $\xi\in \ell^2(\Gamma/\Lambda, \mathcal{K})$. We recall that the map $\mathcal{U}: \ell^2(\Gamma/\Lambda, \mathcal{K})\to \tilde{\Ki}$ given by $[\mathcal{U}\xi](t) = \sigma(c(t, \Lambda)^{-1})\xi(t\Lambda)$ ($\xi\in \ell^2(\Gamma/\Lambda, \mathcal{K}), t\in\Gamma$) defines a unitary equivalence between $(\sigma', \ell^2(\Gamma/\Lambda, \mathcal{K}))$ and $(\tilde{\sigma}, \tilde{\Ki})$.
\item[(3)] (Rieffel induction) Given a unital inclusion of \textup{C}$^*$-algebras $\mathcal{A}\subset\mathcal{B}$ and a conditional expectation $\mathbb{E}_\mathcal{A}: \mathcal{B}\to \mathcal{A}$, there is a standard construction which induces a $*$-representation of $\mathcal{A}$ to a $*$-representation of $\mathcal{B}$. Viewing $\mathcal{B}$ as a left $\mathcal{B}$-module and a right $\mathcal{A}$-module, endow the algebraic tensor product $L = \mathcal{B}\odot_{\mathcal{A}}\mathcal{K}$ with the semi-inner product $\langle \cdot, \cdot\rangle_{\mathbb{E}}: L\times L\to \C$, determined on simple tensors by the formula $
\langle a\otimes \xi, c\otimes \eta\rangle_{\mathbb{E}} = \langle \sigma(\mathbb{E}(c^*a))\xi, \eta\rangle_{\Ki}.$
Setting $I = \{v\in L: \langle v, v\rangle_{\mathbb{E}} = 0\}$, let $_{\Lambda}^\Gamma\Ki$ denote the completion of the vector space quotient $V/I$ with respect to the inner product $\langle \cdot, \cdot\rangle_{\mathbb{E}}$. The formula 
\begin{align*}
([{_\mathcal{A}^\mathcal{B}}\pi](b))(v + I) = bv + I \qquad(v\in L, b\in \mathcal{B})
\end{align*}
determines a $*$-representation $({{_\mathcal{A}^\mathcal{B}}\Ki}, {{_\mathcal{A}^\mathcal{B}}\pi})$ of $\mathcal{B}$. In the particular case where $\mathcal{B} = \textup{C}^*\Gamma$, $\mathcal{A} = \textup{C}^*\Lambda$, and $\mathbb{E}_\Lambda: \textup{C}^*\Gamma\to \textup{C}^*\Lambda$ is the canonical conditional expectation obtained by restricting the support of an element of $\C\Gamma$ to $\Lambda$, we recall that the map $\mathcal{V}:$ ${_\mathcal{A}^\mathcal{B}}\Ki\to \hat{\Ki}$ determined by the formula 
\begin{align}\label{rieffel}
[\mathcal{U}(a\otimes \xi)] = 
\sigma'(a)\xi_e
\end{align}
(where $\xi_e$ is the element of $\tilde{\Ki}$ supported on $\Lambda$ and determined by $\xi_e(e) = \xi$) is a unitary equivalence between $({{_\mathcal{A}^\mathcal{B}}\Ki}, {{_\mathcal{A}^\mathcal{B}}\sigma})$ and $(\tilde{\sigma}, \tilde{\Ki})$.
\end{enumerate}
These three representations will all be identified up to unitary equivalence, and denoted by $(\Ind_{\Lambda}^\Gamma\Ki, \Ind_{\Lambda}^\Gamma\sigma)$.

The following fact is well-known to experts, and will be useful in Section 5. We record a proof for the reader's convenience.

\begin{lem}\label{constant}
    Let $\Gamma\curvearrowright X$ be a minimal action. If $x\in\Gamma$ is such that $\Gamma_x/\Gamma_x^\circ$ is amenable, then $\lambda_{\Gamma/\Gamma_x}\prec\lambda_{\Gamma/\Gamma_y}$, for all $y\in X$. 
\end{lem}
\begin{proof}
    Let $y\in X$. By minimality, we can find a net $g_i\in \Gamma$ such that $g_iy\to x$. Passing to a subnet, we can assume that the net $g_i\Gamma_yg_i^{-1}$ converges to $K$ in ${\rm Sub}(\Gamma)$. It is easy to see that $\Gamma_x^\circ \subset K\subset \Gamma_x$. By \cite[Proposition 3.3]{BeKa20}, we have that $\lambda_{\Gamma/\Gamma_y}\sim \lambda_{\Gamma/K}$. Because $\Gamma_x^\circ$ is co-amenable in $\Gamma_x$, we have that $K$ is co-amenable in $\Gamma_x$. By \cite[Proposition 2.3]{KS2022}, $\lambda_{\Gamma/\Gamma_x}\prec \lambda_{\Gamma/K}$, so that $\lambda_{\Gamma/\Gamma_x}\prec \lambda_{\Gamma/\Gamma_y}$.
\end{proof}

\section{Simplicity and existence of traces}\label{sect3}

Let $\mathcal{A}$ be a unital $\Gamma$-\textup{C}$^*$-algebra, and $\Lambda\leq\Gamma$. Let $\Psi = (\Ki, \sigma, \rho)$ be a covariant representation of the pair $(\Lambda, \mathcal{A})$. Consider the $*$-representation $\Ind_{\Lambda}^\Gamma\rho: \mathcal{A}\to \mathbb{B}(\Ind_{\Lambda}^\Gamma\Ki)$ defined by
\begin{align*}
[\Ind_{\Lambda}^\Gamma\rho(a)]\xi(t) = \rho(t^{-1}\cdot a)\xi(t).
\end{align*}
The next proposition follows from verifying a few routine calculations.
\begin{prop}
The triple $(\Ind_{\Lambda}^\Gamma\Ki, \Ind_{\Lambda}^\Gamma\sigma, \Ind_{\Lambda}^\Gamma\rho)$ is a well-defined covariant representation of the pair $(\Gamma, \mathcal{A})$.
\end{prop}

The covariant representations of $(\Lambda, \mathcal{A})$ correspond to the $*$-representations of the universal crossed product $\mathcal{A}\rtimes \Lambda$. Consider the canonical conditional expectation $\mathbb{E}_\Lambda: \mathcal{A}\rtimes \Gamma\to \mathcal{A}\rtimes \Lambda$ obtained by restricting the support of an element of $C_c(\Gamma, \mathcal{A})$ to $\Lambda$. The following fact is straightforward to check using the intertwining unitary from $(\ref{rieffel})$. 

\begin{prop}
The covariant representation $(\Ind_{\Lambda}^\Gamma\Ki, \Ind_{\Lambda}^\Gamma\sigma, \Ind_{\Lambda}^\Gamma\rho)$ is unitarily equivalent to the induced representation (in the sense of Rieffel) of the $*$-representation $\rho\rtimes\sigma: \mathcal{A}\rtimes \Lambda\to \mathbb{B}(\Ki)$, with respect to the conditional expectation $\mathbb{E}_\Lambda$.
\end{prop}

We use the abbreviation $\Ind_{\Lambda}^\Gamma\Psi = \Ind_{\Lambda}^\Gamma(\Ki, \sigma, \rho)$ for the induced covariant representation $(\Ind_{\Lambda}^\Gamma\Ki, \Ind_{\Lambda}^\Gamma\sigma,\Ind_{\Lambda}^\Gamma\rho)$. We will use this construction to provide more examples of covariant representations of $(\Gamma, X)$ whose crossed products admit unique boundary maps. Throughout the rest of this section, we will use the model (\ref{induction1}) for $\Ind_\Lambda^\Gamma\mathcal{K}$. 

We begin by recalling the definition of a \textit{germinal representation}.

\begin{defn}[{{\cite[Definition 4.1]{KS2022}}}]\label{germ}
Let $ X$ be a compact $\Gamma$-space. A covariant representation $\Phi = (\mathcal{H}, \pi, \rho)$ of $(\Gamma,  X)$ is called \textit{germinal} if, for all $g\in\Gamma$ and all $f\in C( X)$ with $\text{supp}(f)\subset \text{Fix}(g)^{\circ}$, we have
\begin{align*}
\pi(g)\rho(f) = \rho(f).
\end{align*}
\end{defn}

It follows from the next theorem that the \textup{C}$^*$-algebras associated with germinal covariant representations have a rather specific ideal structure.

\begin{thm}[{{\cite[Theorem 4.4]{KS2022}}}]\label{germs}
Let $ X$ be a $\Gamma$-boundary and $\Phi = (\mathcal{H}, \pi, \rho)$ a germinal covariant representation. Set $\mathcal{A} = \textup{C}^*(\pi(\Gamma), \rho(C( X)))\subset \mathbb{B}(\mathcal{H})$. There is a unique boundary map $\varphi\in \textup{BD}(\mathcal{A})$, whose restriction to $\textup{C}^*_\pi\Gamma$ is the unique boundary map $\psi\in \textup{BD}(\textup{C}^*_\pi\Gamma)$.
\end{thm}
In particular (see \cite[Proposition 3.1]{KS2022}), the group \textup{C}$^*$-algebra and crossed product generated by a germinal covariant representation both contain a unique maximal ideal.

\begin{prop}\label{newgerminal}
Let $X$ be a compact $\Gamma$-space, $\Lambda\leq\Gamma$ be a subgroup, and $K\subset X$ a closed, $\Lambda$-invariant subset such that $\Gamma_x^{\circ}\leq \Lambda$ for each $x\in K$. Let $\nu$ be a $\sigma$-finite, $\Lambda$-quasi-invariant Borel measure on $K$, and $(L^2(K, \nu), \kappa, \rho)$ the associated Koopman representation of the dynamical system $(\Lambda,  X)$. The covariant representation $\Ind_{\Lambda}^\Gamma( L^2(K, \nu), \kappa, \rho)$ is a germinal covariant representation of $(\Gamma, X)$.
\end{prop}
\begin{proof}
By \cite[Proposition 4.3]{KS2022}, the Koopman representation $(L^2(X, \nu), \kappa, \rho)$ is a germinal representation. Let $\textup{supp}(f)\subset \text{Fix}(g)^{\circ}$, $t\in\Gamma$, $x\in K$, and $\xi\in \Ind_{\Lambda}^\Gamma L^2(K, \nu)$. If $t^{-1}gt\not\in \Lambda$, then because $\Gamma_x^\circ\leq \Lambda$, we have
\begin{align*}
([\Ind_{\Lambda}^\Gamma\kappa(g)\Ind_{\Lambda}^\Gamma\rho(f)\xi](t))(x) = 0 = ([\Ind_{\Lambda}^\Gamma\rho(f)\xi](t))(x).
\end{align*}
On the other hand, if $t^{-1}gt\in \Lambda$, then we have
\begin{align*}
[\Ind_{\Lambda}^\Gamma\kappa(g)\Ind_{\Lambda}^\Gamma\rho(f)\xi](t) &= [\kappa(t^{-1}gt)\rho(t^{-1}f)]\xi(t).
\end{align*}
But
\begin{align*}
\textup{supp}(t^{-1}f)= t^{-1}\textup{supp}(f)\subseteq t^{-1}\textup{Fix}(g)^{\circ}\subseteq \textup{Fix}(t^{-1}gt)^{\circ},
\end{align*}
so germinality of the Koopman representation implies that
\begin{align*}
[\Ind_{\Lambda}^\Gamma\pi(g)\Ind_{\Lambda}^\Gamma\rho(f)\xi](t) &= [\pi(t^{-1}gt)\rho(t^{-1}f)]\xi(t)\\
&= [\rho(t^{-1}f)]\xi(t)\\
&= [\Ind_{\Lambda}^\Gamma\rho(f)\xi](t).
\end{align*}
This finishes the proof. \end{proof}

\begin{cor}
Follow the notation from Proposition \ref{newgerminal}, and set $N = \textup{ker}(\Gamma\acts \mathcal{X})$. The group \textup{C}$^*$-algebra $\textup{C}^*_{\Ind_\Lambda^\Gamma\kappa}\Gamma$ admits at most one tracial state. If $X$ is a $\Gamma$-boundary, then $\textup{C}^*_{\Ind_\Lambda^\Gamma\kappa}\Gamma$ contains a unique maximal ideal, and admits a tracial state if and only if the action $\Gamma/N\acts X$ is topologically free.
\end{cor}

We now give a family of covariant representations for $(\Gamma, X)$ which are not germinal, but nonetheless generate \textup{C}$^*$-algebras which admit unique boundary maps. Given a subgroup $\Lambda\leq\Gamma$, recall that the \textit{normalizer} of $\Lambda$ is the subgroup $N_\Gamma \Lambda\leq \Gamma$ of all $t\in\Gamma$ for which $t^{-1}\Lambda t = \Lambda$. 

\begin{defn}\label{rc}
Let $\Lambda\leq\Gamma$, and $\chi: \Lambda\to \T$ a character. We say $\chi$ is \textit{relatively central} if $\chi(t^{-1}ht) = \chi(h)$, for all $h\in \Lambda$ and $t\in N_\Gamma \Lambda$; we say $\chi$ is \textit{centrally extendable} if it is the restriction to $\Lambda$ of a conjugation-invariant function $\varphi:\Gamma\to \C$.
\end{defn}
Note that a character $\chi$ on $\Lambda$ is relatively central if and only if it is $N_\Gamma\Lambda$-invariant. 

\begin{exl}
Let $\Lambda\leq\Gamma$, and $\chi$ a character on $\Lambda$.
\begin{enumerate}
\item If $\chi$ is centrally extendable, then $\chi$ is relatively central.
\item If $\chi$ is invariant under $\Aut(\Lambda)$, then $\chi$ is relatively central.
\item If $\Lambda$ is self-normalizing, then $\chi$ is relatively central.
\end{enumerate}
\end{exl}

Given a relatively central character $\chi$, we will use $\tilde{\chi}: \Gamma\to \C$ to denote any extension of $\chi$ that is constant on the conjugacy classes $\textup{Cl}(h)$ (for $h\in\Lambda$).

\begin{prop}\label{qg1}
Let $X$ be a compact $\Gamma$-space, and $x\in X$. If $\,\Gamma_x^\circ\leq \Lambda\leq \Gamma_x$, $\chi: \Lambda\to \T$ is a relatively central character, and $\rho: C(X)\to \mathbb{B}(\C)$ is the evaluation map $\rho(f) = f(x)$, then 
\begin{align*}
\Ind_{\Lambda}^\Gamma\chi(g)\Ind_{\Lambda}^\Gamma\pi(f) = \tilde{\chi}(g)\Ind_{\Lambda}^\Gamma\pi(f),
\end{align*}
whenever $g\in\Gamma$, $f\in C(X)$, and $\textup{supp}(f)\subset \textup{Fix}(g)^\circ$.
\end{prop}
\begin{proof}
We may assume that $\text{Fix}(g)^{\circ}$ is nonempty; otherwise, the equality holds trivially. Let $t\in\Gamma$. If $t^{-1}gt\not\in \Lambda$, then $f(g^{-1}t\cdot x) = f(t\cdot x) = 0$, so automatically $\Ind_{\Lambda}^\Gamma\chi(g)\Ind_{\Lambda}^\Gamma\rho(f)\xi(t) = 0 = \tilde{\chi}(g)\Ind_{\Lambda}^\Gamma\rho(f)(t)$ for all $\xi\in\Ind_{\Lambda}^\Gamma\C$. Now suppose $t^{-1}gt\in \Lambda$. For all $\xi\in\Ind_{\Lambda}^\Gamma\C$, we have
\begin{align*}
\Ind_{\Lambda}^\Gamma\chi(g)\Ind_{\Lambda}^\Gamma\rho(f)\xi(t) = f(g^{-1}t\cdot x)\xi(g^{-1}t) = f(t\cdot x)\chi(t^{-1}gt)\xi(t) = \tilde{\chi}(g)\Ind_{\Lambda}^\Gamma\rho(f)\xi(t).
\end{align*} This finishes the proof. \end{proof}

\begin{cor}\label{uniqueness}
Consider the induced covariant representation from Proposition \ref{qg1}. Set $\mathcal{A} = \textup{C}^*(\Ind_{\Lambda}^\Gamma\chi(\Gamma), \Ind_{\Lambda}^\Gamma\rho(C( X)))\subset \mathbb{B}(\mathcal{H})$. If $\varphi\in \textup{BD}(\mathcal{A})$, then $\varphi(\Ind_{\Lambda}^\Gamma\chi(g)) = \tilde{\chi}(g)\textbf{\textup{1}}_{\Delta_g}$, for every $g\in\Gamma$. In particular, there is a unique boundary map $\varphi\in \textup{BD}(\mathcal{A})$, whose restriction to $\textup{C}^*_{\Ind_{\Lambda}^\Gamma\chi}\Gamma$ is the unique boundary map $\psi\in \textup{BD}(\textup{C}^*_{\Ind_{\Lambda}^\Gamma\chi}\Gamma)$.
\end{cor}

\begin{proof}
We mimic the proof of \cite[Proposition 4.4]{KS2022}. By \cite[Proposition 3.2]{KS2022}, it suffices to check that $\varphi(\Ind_{\Lambda}^\Gamma\chi(g))(x) = \tilde{\chi}(g)$, whenever $x\in \mathfrak{b}_X^{-1}(\text{Fix}(g)^\circ)$. We are done whenever $\text{Fix}(g)^\circ =\emptyset$, so assume $\text{Fix}(g)^\circ\neq\emptyset$. Fix $x\in \mathfrak{b}_X^{-1}(\text{Fix}(g)^\circ)$. Find a function $f\in C(X)$ such that $\textup{supp}(f)\subset \textup{Fix}(g)^\circ$ and $f(\mathfrak{b}_X(x)) = 1$. By the multiplicative domain principle for ucp maps \cite[Theorem 3.18]{Paulsen2003}, and Theorem \ref{rigidity}, we have
\begin{align*}
\varphi(\Ind_{\Lambda}^\Gamma\chi(g))(x) &= \varphi(\Ind_{\Lambda}^\Gamma\chi(g))(x)(f\circ \mathfrak{b}_X)(x) \\
&= \varphi(\Ind_{\Lambda}^\Gamma\chi(g))(x)\varphi(\rho(f))(x) \\
&= \varphi(\Ind_{\Lambda}^\Gamma\chi(g)\rho(f))(x) \\
&= \varphi(\tilde{\chi}(g)\rho(f))(x)\\
&= \tilde{\chi}(g)\varphi(\rho(f))(x) \\
&= \tilde{\chi}(g)(f\circ \mathfrak{b}_X)(x) \\
&= \tilde{\chi}(g).
\end{align*}
It follows from the multiplicative domain principle that the $\Gamma$-ucp $\varphi: \mathcal{A}\to C(\partial_F\Gamma)$ determined by $\varphi(\rho(f)) = f\circ\mathfrak{b}_X$ and $\varphi(\Ind_{\Lambda}^\Gamma\chi(g)) = \tilde{\chi}(g)\textbf{1}_{\Delta_g}$ ($f\in C(X)$, $g\in\Gamma$, and $\tilde{\varphi}$ a $\Gamma$-ucp extension of $\varphi$) is the unique boundary map on $\mathcal{A}$, whose restriction to $\textup{C}^*_{\Ind_{\Lambda}^\Gamma\chi}\Gamma$ is the unique $\Gamma$-ucp on $\textup{C}^*_{\Ind_{\Lambda}^\Gamma\chi}\Gamma$.
\end{proof}

\begin{cor}
Follow the notation from Corollary \ref{uniqueness}. The faithful kernel $I_{\varphi}$ of the unique boundary map $\varphi\in \textup{BD}(\mathcal{A})$ is the unique maximal ideal of $\mathcal{A}$. Furthermore, the intersection $I_{\varphi}\cap \textup{C}^*_{\Ind_{\Lambda}^\Gamma\chi}\Gamma$ is the unique maximal ideal of $\textup{C}^*_{\Ind_{\Lambda}^\Gamma\chi}\Gamma$.
\end{cor}

\begin{proof}
This follows immediately from \cite[Proposition 3.1]{KS2022}: by Theorem \ref{rigidity}, $I_{\varphi}$ is a two-sided ideal in $\mathcal{A}$ as well, so the same argument applies.
\end{proof}

We now consider the existence of traces for these $\textup{C}^*$-algebras. With the notation from Proposition \ref{qg1}, let $N = \ker(\Gamma\acts X)$, and let $\chi_0$ denote the 0-extension of $\chi|_N$ to $\Gamma$. Note that $\chi_0$ defines a trace on $\Gamma$. Let $(\Hi_\chi, \pi_\chi, \xi_\chi)$ denote the GNS representation associated to $\chi_0$.

\begin{thm}\label{traces} Consider the induced covariant representation from Proposition \ref{qg1}. The following conditions are equivalent:
\begin{enumerate}
\item $\textup{C}^*_{\Ind_{\Lambda}^\Gamma\chi}\Gamma$ admits a trace;
\item The action $\Gamma/N\acts X$ is topologically free;
\item Whenever $\sigma$ is a unitary representation of $\Gamma$ and $\sigma \prec \Ind_{\Lambda}^\Gamma\chi$, we have 
$\pi_\chi \prec \sigma$;
\item $\pi_{\chi}\prec\Ind_{\Lambda}^\Gamma\chi$.
\end{enumerate}
Moreover, if any of the above conditions hold, then $\chi_0$ is the unique trace (and boundary map) on $\textup{C}^*_{\Ind_{\Lambda}^\Gamma\chi}\Gamma$.
\end{thm}

\begin{proof}
We follow the argument in \cite[Theorem 4.8]{KS2022}.

$1)\implies 2)$ Let $\varphi\in \textup{BD}(\textup{C}^*_{\Ind_{\Lambda}^\Gamma\chi}\Gamma)$. If $\textup{C}^*_{\Ind_{\Lambda}^\Gamma\chi}\Gamma$ admits a trace $\tau$, then by Theorem \ref{uniqueness}, $\tau(g) = \tilde{\chi}(g)\textbf{1}_{\Delta_g}$ is a constant function for every $g\in\Gamma$. Thus, if $g\not\in N$, then $\Delta_g = \emptyset$; so $\textup{Fix}(g)^\circ = \emptyset$.

$2)\implies 3)$ Let $\varphi\in \textup{BD}(\textup{C}^*_{\Ind_{\Lambda}^\Gamma\chi}\Gamma)$. By Theorem \ref{uniqueness}, $\varphi(\Ind_{\Lambda}^\Gamma\chi(g)) = \tilde{\chi}(g)\textbf{1}_{\Delta_g}$. By the topological freeness of the action $\Gamma/N\acts X$, we automatically have $\varphi(\Ind_{\Lambda}^\Gamma\chi(g)) = \chi(g)$ for $g\in N$, and $\varphi(\Ind_{\Lambda}^\Gamma\chi(g)) = 0$ for $g\not\in N$: thus, $\varphi = \chi_0$.

Now suppose that $\sigma\prec \Ind_{\Lambda}^\Gamma\chi$. By the uniqueness of $\varphi$, the group $\textup{C}^*$-algebra $\textup{C}^*_\sigma\Gamma$ also admits a unique boundary map $\varphi'$, on which $\varphi'(\sigma(g)) = \varphi(\Ind_{\Lambda}^\Gamma\chi(g)) = \chi_0(g)$, for every $g\in\Gamma$. We have just shown that the trace $\chi_{0}$ extends to the group $\textup{C}^*$-algebra $\textup{C}^*_\sigma\Gamma$, so that $\pi_\chi\prec \sigma$, as desired.

$3)\implies 4)$ is straightforward.

$4)\implies 1)$ The group $\textup{C}^*$-algebra $\textup{C}^*_{\pi_\chi}\Gamma$ admits the trace $\chi_0$; thus, if $\pi_\chi\prec \Ind_{\Lambda}^\Gamma\chi$, then $\Ind_{\Lambda}^\Gamma\chi$ also admits the trace $\chi_0$.
\end{proof}

We end the section by considering the simplicity of the $\textup{C}^*$-algebras associated with the covariant representation from Proposition \ref{qg1}.

\begin{prop}\label{simple}
Consider the induced covariant representation from Proposition \ref{qg1}. If $x\in X_0$, then $\textup{C}^*_{\Ind_{\Lambda}^\Gamma\chi}\Gamma$ and $\textup{C}^*(\Ind_{\Lambda}^\Gamma\chi(\Gamma), \Ind_{\Lambda}^\Gamma\rho(C(X)))$ (respectively) are simple.
\end{prop}
\begin{proof}
Suppose $\theta\prec \Ind_{\Lambda}^\Gamma\chi$. There is a unique boundary map $\varphi$ on $\textup{C}^*_\theta\Gamma$, determined on $\Gamma$ by the formula $\varphi(\theta(g)) = \tilde{\chi}(g)\textbf{1}_{\Delta_g}$. Find $y\in \partial_F\Gamma$ with $\mathfrak{b}_{X}(y) = x$. Because $x\in X_0$, we have $\Gamma_x^\circ = \{h\in\Gamma: y\in \Delta_h\}$. Thus, composing $\varphi$ with $\delta_y$, we obtain a state $\psi: \textup{C}^*_{\theta}\Gamma\to \C$ such that $\psi(g) = \delta_y(\tilde{\chi}(g)\textbf{1}_{\Delta_g}) = \tilde{\chi}(g)\textbf{1}_\Lambda(g),$ for all $g\in\Gamma$. We have just shown that $\psi|_{\Gamma}$ is the 0-extension of $\chi$ to $\Gamma$; that is, the positive-definite function associated to $\Ind_{\Lambda}^\Gamma\chi$. It follows that $\Ind_{\Lambda}^\Gamma\chi\sim\theta$, so $\textup{C}^*_{\Ind_{\Lambda}^\Gamma\chi}\Gamma$ is simple. 

The proof that the \textup{C}$^*$-algebra generated by $\Ind_{\Lambda}^\Gamma\chi(\Gamma)\cup \Ind_{\Lambda}^\Gamma\rho(C(X))$ is simple is similar: if $(\pi, \rho')\prec (\Ind_{\Lambda}^\Gamma\chi, \Ind_{\Lambda}^\Gamma\rho)$, then by Corollary \ref{uniqueness} there is a unique boundary map $\varphi$ on C$^*_{\pi\times \rho'}(\Gamma, X)$, determined by $\varphi(\rho'(f)) = f\circ \mathfrak{b}_X$ and $\varphi(\pi(g)) = \tilde{\chi}(g)\textbf{1}_{\Delta_g}$ ($f\in C(X), g\in \Gamma$). Fixing $y\in \partial_F\Gamma$ with $\mathfrak{b}_X(y) = x$, and considering $\psi = \delta_y\circ\varphi$, the argument in the previous paragraph shows that $\psi(\pi(g)) = \tilde{\chi}(g)\textbf{1}_\Lambda(g)$. Thus, by the multiplicative domain principle,
\begin{align*}
\psi(\rho'(f)\pi(g)) = f(y)\tilde{\chi}(g)\textbf{1}_\Lambda(g) = \langle \Ind_\Lambda^\Gamma\rho(f)\Ind_\Lambda^\Gamma\chi(g)\delta_\Lambda, \delta_\Lambda\rangle,
\end{align*}
where $\delta_\Lambda\in \Ind_\Lambda^\Gamma$ is the cyclic vector for $(\Ind_{\Lambda}^\Gamma\chi, \Ind_{\Lambda}^\Gamma\rho)$ supported on $\Lambda$ and determined by $\delta_\Lambda(e) = 1$. The uniqueness of the GNS construction implies $(\Ind_{\Lambda}^\Gamma\chi, \Ind_{\Lambda}^\Gamma\rho)\prec (\pi, \rho')$, as desired.
\end{proof}

\section{Selflessness and pure infiniteness}\label{sect4}

In this section, we provide sufficient conditions for the selflessness of the $\textup{C}^*$-probability space $(\textup{C}^*_{\Ind_{\Lambda}^\Gamma\chi}\Gamma,\varphi)$, where $\varphi$ is the canonical vector state $
\varphi(T) = \langle T\delta_\Lambda, \delta_\Lambda\rangle$
associated with the vector $\delta_{\Lambda}\in \Ind_{\Lambda}^\Gamma\CC$. Throughout this section, we use the model (\ref{induction2}) for unitary representations of $\Gamma$ and covariant representations of $(\Gamma, X)$ induced from a subgroup $\Lambda\leq \Gamma$. We fix a left cocycle $c: \Gamma\times \Gamma/H\to H$ for this model such that $c(h, H) = h$ for all $h\in H$. This choice of left cocycle implies the useful formula
\begin{align*}
\varphi(\Ind_{\Lambda}^\Gamma\chi(g))=\begin{cases}
\chi(g),
    &g\in \Lambda;\\[2mm]
0,
    &g\notin \Lambda.
\end{cases}
\end{align*}

The following lemma is inspired by the work of Ozawa \cite[Section 8]{Oz25}.

\begin{lem}\label{PHPlem}
    Let $X$ be an extreme $\Gamma$-boundary with $\#X>2$. For every $x\in X_0$, every finite subset $F\subset \Gamma\setminus \{e\}$, and every positive integer $n\in\mathbb N$, there exist elements $t_i\in \Gamma$ and subsets $B_i^+, B_i^-\subset \Gamma/\Gamma_x$, $1\leq i\leq n$ satisfying the following properties:
    \begin{enumerate}
        \item[\textup{(\textit{i})}]\label{phpi} $\{B_i^+,B_i^-: 1\leq i\leq n\}$ are mutually disjoint;
        \item[\textup{(\textit{ii})}]\label{phpii} $t_i(\Gamma/\Gamma_x\setminus B_i^-)\subset B_i^+$ for every $1\leq i\leq n$;
        \item[\textup{(\textit{iii})}]\label{phpiii} $g(B_j^+\cup B_j^-)\cap (B_i^+\cup B_i^-)=\emptyset$ for every $g\in F\setminus \Gamma_x$ and every $(i,j)\in \{1,\ldots,n\}^2$, and
        \item[\textup{(\textit{iv})}]\label{phpiv} $gk\Gamma_x=k\Gamma_x$, for every $g\in F\cap \Gamma_x$, $k\Gamma_x\in B_i^+\cup B_i^-$, and $1\leq i\leq n$.
    \end{enumerate}
\end{lem}

\begin{proof}
Recall that, since $x\in X_0$, we have $\Gamma_x=\Gamma_x^\circ$. Let $F\subset \Gamma\setminus \{e\}$ be a finite subset. Since $gx\not=x$ for all $g\in F\setminus \Gamma_x$, there are open sets $V_0$ and $\{V_g\}_{g\in F\setminus \Gamma_x}$ such that $V_0\cap V_g=\emptyset$, $x \in V_0$, and $gx \in V_g$ for all $g\in F\setminus \Gamma_x$. We set
\begin{align*}
      U = V_0 \cap \bigcap_{g\in F\setminus \Gamma_x} g^{-1}V_g,
\end{align*}
   and note that $x\in U$ and $U \cap gU = \emptyset$ for all $g\in F\setminus \Gamma_x$. Since $F\cap \Gamma_x\subset \Gamma_x=\Gamma_x^\circ$ is finite, we can find an open neighborhood $x\in V$ such that every element in $F\cap \Gamma_x$ acts trivially on $V$. By replacing $U$ with $U\cap V$, we will now assume that every element in $F\cap \Gamma_x$ acts trivially on $U$.

Since $X$ is perfect, the open set $U$ is infinite. Hence, for $1\leq i\leq n$, we can find non-empty open disjoint subsets $U_i^+, U_i^-\subset U$ and group elements $t_i\in \Gamma$ such that $t_i(X\setminus U_i^-)\subset U_i^+$. Set $A_i^\pm:=\{h\in \Gamma: hx\in U_i^\pm\}\subset \Gamma$ and $B_i^\pm:=A_i^\pm/\Gamma_x\subset \Gamma/\Gamma_x$. We now verify that the sets $B_i^+, B_i^-$ satisfy the properties claimed above.
\begin{enumerate}
    \item[(\emph{i})] If $i, j\in \{1, ..., n\}$, $k, \ell\in \{+, -\}$, and $g\Gamma_x\in B_i^k\cap B_j^\ell$, then there exist $k,h\in \Gamma$ such that 
    \begin{align*}
    hx=gx=kx\in U_i^k\cap U_j^\ell,
    \end{align*}
    which is possible if and only if $k = \ell$ and $i = j$.
    \item[(\emph{ii})] One has $(\Gamma/\Gamma_x\setminus B_i^-)\subset  \{h\in \Gamma: hx\notin U_i^-\}/\Gamma_x$, which easily implies the conclusion.
    \item[(\emph{iii})] Suppose that $g\in F\setminus \Gamma_x$ and $t\Gamma_x\in g(B_j^+\cup B_j^-)\cap (B_i^+\cup B_i^-)$. Find $h_1,h_2\in \Gamma_x$ such that $th_1\in g(A_j^+\cup A_j^-)$ and $th_2\in (A_i^+\cup A_i^-) $. Since $th_1x=tx= th_2x$, we have 
    \begin{align*}
    tx\in g(U_j^+\cup U_j^-)\cap (U_i^+\cup U_i^-)\subset gU\cap U,
    \end{align*}
    which is impossible, since $gU\cap U=\emptyset$.
    \item[(\emph{iv})] Let $g\in F\cap \Gamma_x$ and $k\Gamma_x\in B_i^+\cup B_i^-$. Because $kx\in U_i^+\cup U_i^-$ and $g$ fixes $U$, we have $k^{-1}gkx = x$, so $gk\Gamma_x = k\Gamma_x$.
\end{enumerate}
\end{proof}

We now recall the Property PHP for $\textup{C}^*$-probability spaces. This definition is taken from Hayes-Kunnawalkam Elayavalli-Patchell-Robert \cite{HKEPR25}, where it was originally formulated for inclusions. It is based on the work of Ozawa \cite{Oz25}, who introduced it for groups, and it provides a reliable way of proving selflessness. We point out that while in \cite{HKEPR25} selflessness is formulated for inclusions, we only use it for individual $\textup{C}^*$-algebras.

\begin{defn}[Property PHP \cite{HKEPR25}]\label{php}
	Let $(\mathcal{A},\tau)$ be a unital $\textup{C}^*$-probability space and assume that $\mathcal{A}\subseteq \mathbb B(\mathcal{H})$. The pair $(\mathcal{A},\tau)$ has Property PHP if for all finite sets  $F\subset \ker(\tau)$, all $\epsilon > 0$, and all $n\in \N$, there exist $u_i,P_{i},P_i^+$, for $i=1,\ldots,n$ such that  $u_i\in  U(\mathcal{A})$, 
	$P_i,P_i^+\in \mathbb B(\mathcal{H})$ are projections with $P_i^+\leq P_i$,  and the following hold:
	\begin{enumerate}
		\item[(a)]\label{phpa} $(P_i)_{i=1}^n$ is a pairwise orthogonal family of projections;
		\item[(b)]\label{phpb} $u_iP_i^+u_i^*\leq P_i^+$ and $u_i(1-P_i)u_i^* \leq  P_i^+$  for all $i$;
		\item[(c)]\label{phpc} $\|P_ixP_j\|<\epsilon$ for all $x\in F$ and all $i,j$.
	\end{enumerate}
\end{defn}

\begin{thm}[{{\cite[Theorem 3.3]{HKEPR25}}}]
	If $(\mathcal{A},\tau) \subset \mathbb B(\mathcal{H})$ has Property \textup{PHP}, then $ (\mathcal{A},\tau)$ is a selfless $\textup{C}^*$-probability space.
\end{thm}

\begin{thm}\label{selfthm}
    Let $X$ be an extreme $\Gamma$-boundary with $\#X>2$. For every $x\in X_0$ and every centrally extendable character $\chi:\Gamma_x\to\T$, the pair $(\textup{C}^*_{\Ind_{\Gamma_x}^\Gamma\chi}\Gamma, \varphi)$ has Property \textup{PHP}.
\end{thm}

\begin{proof}
    As mentioned in \cite[Definition 3.1]{HKEPR25}, it is enough to choose the finite sets $F\subset \ker{\varphi}$ from a subset with dense linear span in $\ker(\varphi)$, so we will consider a finite subset
    \begin{equation}\label{sett}
        F\subset\{\Ind_{\Gamma_x}^\Gamma\chi(h)-\chi(h)1:h\in \Gamma_x\}\cup\{\Ind_{\Gamma_x}^\Gamma\chi(g):g\in \Gamma\setminus \Gamma_x\},
    \end{equation}
   and let
    \begin{align*}
   F'=\{g\in \Gamma : \Ind_{\Gamma_x}^\Gamma\chi(g)\in F\text{ or }\Ind_{\Gamma_x}^\Gamma\chi(g)-\chi(g)1\in F\}
   \end{align*}
be the corresponding subset of $\Gamma$. Consider the $\Gamma_x$-equivariant representation {$\rho: \ell^\infty(\Gamma/\Gamma_x)\to \mathbb{B}(\C)$} defined by $\rho(f)z = f(\Gamma_x)z$ ($f\in \ell^\infty(\Gamma/\Gamma_x)$, $z\in \C$). We note that the induced covariant representation $\Ind_{\Gamma_x}^\Gamma(\C, \chi, \rho)$ satisfies the formula $[\Ind_{\Gamma_x}^\Gamma\rho(f)]\xi(t\Gamma_x) = f(t\Gamma_x)\xi(t\Gamma_x)$.

Fix $n\in \N$. For $1\leq i\leq n$, let $t_i\in \Gamma$ and $B_i^+, B_i^-\subset \Gamma/\Gamma_x$ be the group elements and subsets obtained from applying Lemma \ref{PHPlem} to the set $F$. We claim that the unitaries $u_i$ and orthogonal projections $P_i^+, P_i\in \mathbb{B}(\Ind_{\Gamma_x}^\Gamma\C)$ defined by
\begin{align*}
u_i=\Ind_{\Gamma_x}^\Gamma\chi(t_i), \qquad P_i^+=\Ind_{\Gamma_x}^\Gamma\rho(1_{B_i^+}), \quad \textup{and}\quad P_i=\Ind_{\Gamma_x}^\Gamma\rho(1_{B_i^+} + 1_{B_i^-})
\end{align*}
satisfy Definition \ref{php}.
\begin{enumerate}
    \item[(a)] Property PHP(a) follows immediately from Lemma \ref{PHPlem}(\hyperref[phpi]{i}).
    \item[(b)] Property PHP(b) follows from Lemma \ref{PHPlem}(\hyperref[phpii]{ii}) and covariance of the induced representation.
    \item[(c)] Given $g\in F'$, there are two cases to consider in the proof of Property PHP(c): $g\not\in \Gamma_x$ and $g\in\Gamma_x$. If $g\not\in \Gamma_x$, then covariance of the induced representation and Lemma \ref{PHPlem}(\hyperref[phpiii]{iii}) implies
    \begin{align*}
P_i\Ind_{\Gamma_x}^\Gamma\chi(g)P_j = \Ind_{\Gamma_x}^\Gamma\chi(g)\Ind_{\Gamma_x}^\Gamma\rho(1_{g(B_i^+\cup B_i^-)\cap (B_j^+\cup B_j^-)}) = 0.
    \end{align*}
    On the other hand, if $g\in \Gamma_x$, then because $g(B_i^+\cup B_i^-) = B_i^+\cup B_i^-$ (Lemma \ref{PHPlem}(\hyperref[phpiv]{iv})), covariance of the induced representation implies that $P_i\Ind_{\Gamma_x}^\Gamma\chi(g) = \Ind_{\Gamma_x}^\Gamma\chi(g)P_i$. Thus, we need only to show that $P_i(\Ind_{\Gamma_x}^\Gamma\chi(g) - \chi(g)1)P_j = 0$ when $i = j$. If $k\in\Gamma\setminus A_i^+\cup A_i^-$, then automatically $P_i(\Ind_{\Gamma_x}^\Gamma\chi(g) - \chi(g)1)P_i\delta_{k\Gamma_x} = 0$, so we need only to consider $k\in A_i^+\cup A_i^-$. For such $k$, Lemma \ref{PHPlem}(\hyperref[phpiv]{iv}) implies that
    \begin{align*}
P_i(\Ind_{\Gamma_x}^\Gamma\chi(g)-\chi(g)1)P_i\delta_{k\Gamma_x}&=  P_i(\Ind_{\Gamma_x}^\Gamma\chi(g)-\chi(g)1)\delta_{k\Gamma_x}\\
&= P_i(\chi(c(g^{-1}, k\Gamma_x)^{-1})\delta_{gk\Gamma_x} - \chi(g)\delta_{k\Gamma_x})\\
&= \chi(c(g, k\Gamma_x))\delta_{k\Gamma_x} - \chi(g)\delta_{k\Gamma_x},
    \end{align*}
    where the last line follows because $c(g^{-1}, gk\Gamma_x)^{-1} = c(g, k\Gamma_x)$ and $gk\Gamma_x = k\Gamma_x$. Now let $\tilde{\chi}: \Gamma\to \C$ be a central extension of $\chi$. Because $k^{-1}gk\in \Gamma_x$, the cocycle relation implies that $c(k, \Gamma_x)^{-1}c(g, k\Gamma_x)c(k, \Gamma_x) = c(k^{-1}gk, \Gamma_x) = k^{-1}gk$. Thus,
    \begin{align*}
\chi(c(g, k\Gamma_x)) = \tilde{\chi}(c(k, \Gamma_x)^{-1}c(g, k\Gamma_x)c(k, \Gamma_x)) = \tilde{\chi}(k^{-1}gk) = \chi(h),
    \end{align*}
    so $P_i(\Ind_{\Gamma_x}^\Gamma\chi(g)-\chi(g)1)P_i\delta_{k\Gamma_x}= 0$, as desired.
\end{enumerate}
\end{proof}

\begin{cor}\label{selfcor}
    Let $X$ be an extreme $\Gamma$-boundary with $\#X>2$. If $\Gamma_x/\Gamma_x^\circ$ is amenable for every $x\in X$, then the $\textup{C}^*$-algebra $\textup{C}^*_{\lambda_{\Gamma/\Gamma_x}}\Gamma$ is selfless with respect to its canonical state. 
\end{cor}

\begin{proof}
    This follows from Lemma \ref{constant} and Theorem \ref{selfthm}.
\end{proof}

\section{An application: C*-algebras associated with Thompson-Stein groups}\label{application}

We now apply our results to study the operator algebraic aspects of Thompson-like groups. The class of groups we study was introduced by Stein in \cite{St92}; our exposition follows \cite{BiSt,Ta24}.

\begin{defn}[Thompson-Stein groups]
Let $\Lambda \subset \mathbb{R}$ be a non-trivial multiplicative subgroup of $\R^+$, and $\Gamma\subset \mathbb{R}$ a $\mathbb{Z}\Lambda$-submodule such that $\Lambda \cdot \Gamma=\Gamma$. Fix $1 \leq \ell \in \Gamma$. 
\begin{itemize}
    \item We denote by $V(\Gamma,\Lambda, \ell)$ the right continuous, piecewise linear bijections of $[0,\ell]$ with slopes in $\Lambda$ and finitely many discontinuities in $\Gamma$. 
    \item We denote by $T(\Gamma,\Lambda,\ell) \subset V(\Gamma,\Lambda,\ell)$ the group of continuous piecewise linear bijections of $[0,\ell]/{\{0\sim\ell\}}$, with non-differentiable points in $\Gamma$ and slopes in $\Lambda$.
    \item We denote by $F(\Gamma,\Lambda, \ell) \subset T(\Gamma,\Lambda,\ell) \subset V(\Gamma,\Lambda,\ell)$ the space of piecewise linear homeomorphisms of $[0,\ell]$ with non-differentiable points in $\Gamma$ and slopes in $\Lambda$.
\end{itemize}\label{steins group}
\end{defn}
We note that Thompson's group $V$ is recovered as $V=V( \mathbb{Z}[1/2], \langle 2 \rangle,1 )$, and the Higman-Thompson groups correspond to $V_{n,r}=V(\mathbb{Z}[1/n], \langle n \rangle , r)$, for $n,r \in \mathbb{N}$ (see \cite{Hi74}). Many generalizations of the Higman-Thompson groups fit into this framework, including Cleary's irrational slope Thompson groups \cite{Cl95,Cl20}. We refer the reader to \cite{St92,Ta24} for several other examples.

The group $T(\Gamma,\Lambda,\ell)$ naturally acts on $\mathbb T\cong \R/\ell\Z$ by homeomorphisms. The stabilizer of $0$ with respect to this action is canonically isomorphic to $F(\Gamma,\Lambda,\ell)$. In the next result, we denote by $\pi_\ell:\R\to \R/\ell\Z$ the canonical quotient map.

\begin{prop}\label{TSstuff}
    Let $T(\Gamma,\Lambda,\ell)\curvearrowright \R/\ell\Z$ be the canonical action of the Thompson-Stein group $T(\Gamma,\Lambda,\ell)$.
    \begin{itemize}
        \item[\textup{(\textit{i})}] $\R/\ell\Z$ is an extreme $T(\Gamma,\Lambda,\ell)$-boundary.
        \item[\textup{(\textit{ii})}] The group $ T(\Gamma,\Lambda,\ell)_x/T(\Gamma,\Lambda,\ell)_x^\circ$ is abelian for every $x\in\R/\ell\Z$.
    \end{itemize}
\end{prop}
\begin{proof}
    The proof of (\emph{i}) is a direct modification of the Bieri-Strebel criterion \cite[ Theorem A.1]{St92}. Indeed, let $K\subset \mathbb T$ be a compact proper subset, and let $U\subset \mathbb T$ be open and non-empty. Since $\Lambda$ is non-trivial, $\Gamma\subset \R$ is dense. Thus, there exist $a,b\in\Gamma$ such that $0<a<b<\ell$ and $\pi_\ell([a,b])\subset \T\setminus K$; in particular, $K\subset \pi_\ell([b,a+\ell])$. For convenience, we set $L=a+\ell-b\in\Gamma$. Because $U$ is open, we can also find $p\in\Lambda$ and $x\in\Gamma$ such that $pL<b-a$ and $\pi_\ell([x,x+pL])\subset U$.

      Now, we define $G:[b,b+\ell]\to [x,x+\ell]$ by
      $$
      G(t)=\begin{cases}
x+p(t-b),
    &b\leq t\leq a+\ell,\\[2mm]
x+pL+p^{-1}(t-a-\ell),
    &a+\ell\leq t\leq a+\ell+pL,\\[2mm]
x+pL+L+(t-a-\ell-pL),
    &a+\ell+pL\leq t\leq b+\ell.
\end{cases}
      $$
      An $\ell$-periodic extension of $G$ to all of $\R$ satisfies the required properties: $G$ is an increasing piecewise linear homeomorphism with breakpoints in $\Gamma$ and slopes $1,p,p^{-1}\in\Lambda$. Furthermore, it maps 
      $$
      G(K)\subset G(\pi_\ell([b,a+\ell]))\subset \pi_\ell([x,x+pL])\subset U.
      $$
      
    We now prove (\emph{ii}). Note that the map $T(\Gamma,\Lambda,\ell)_x\to \Lambda\times \Lambda$ that sends a piecewise linear homeomorphism to its left/right derivatives at $x$, $G\mapsto (G'_{-}(x),G'_{+}(x))$, is a group homomorphism due to the chain rule. If $G'_{-}(x)=G'_{+}(x)=1$, then by the Mean Value Theorem \cite[Theorem 29.5]{Ross13}, $G$ is locally equal to $G(t)=t+c$ for some $c\in \R$; but $G$ fixes $x$, so it is locally equal to the identity map, and thus $G\in T(\Gamma,\Lambda,\ell)_x^\circ$. We conclude that $ T(\Gamma,\Lambda,\ell)_x/T(\Gamma,\Lambda,\ell)_x^\circ$ faithfully embeds into the abelian group $\Lambda\times \Lambda$. \end{proof}

Noting once again that $T(\Gamma,\Lambda,\ell)_0\cong F(\Gamma,\Lambda,\ell)$, Proposition \ref{TSstuff} allows us to apply the results obtained in Theorem \ref{selfcor} to the extreme boundary action $T(\Gamma,\Lambda,\ell)\curvearrowright \R/\ell\Z$.

\begin{cor}\label{TScor}
    For $\Gamma,\Lambda,\ell$ as above, set $F=F(\Gamma,\Lambda,\ell)$ and $T=T(\Gamma,\Lambda,\ell)$. Then, the $\textup{C}^*$-algebra $\textup{C}^*_{\lambda_{T/F}}T$ is simple and purely infinite.
\end{cor}


\section*{AI statement}

No generative AI or large language model was used in the production of this text or its mathematical content.

\section*{Acknowledgments}

F.F.\ gratefully acknowledges support from the Simons Foundation Dissertation Fellowship SFI-MPS-SDF-00015100. The authors wish to thank Professor Ben Hayes and Professor Mehrdad Kalantar for their guidance and interesting discussions surrounding the present topic. The authors also wish to thank the 2026 UK Operator Algebras conference (UKOA) and Sabhal M\`or Ostaig for hosting the initial part of this research project.

\printbibliography

@article{KK14,
author = {M. Kalantar and M. Kennedy},
TITLE = {Boundaries of reduced {$C^*$}-algebras of discrete groups},
   JOURNAL = {J. Reine Angew. Math.},
  FJOURNAL = {Journal f\"ur die Reine und Angewandte Mathematik. [Crelle's
              Journal]},
    VOLUME = {727},
      YEAR = {2017},
     PAGES = {247--267},
}

@misc{Oz25,
      title={Proximality and selflessness for group {\Cs}-algebras}, 
      author={N. Ozawa},
      year={2025},
      eprint={2508.07938},
      archivePrefix={arXiv},
      primaryClass={math.OA},
    %note={preprint, \url{https://arxiv.org/abs/2508.07938}},
}

@article {AGKEP25,
    AUTHOR = {Amrutam, T. and Gao, D. and Kunnawalkam Elayavalli,
              S. and Patchell, G.},
     TITLE = {Strict comparison in reduced group {\Cs}-algebras},
   JOURNAL = {Invent. Math.},
  FJOURNAL = {Inventiones Mathematicae},
    VOLUME = {242},
      YEAR = {2025},
    NUMBER = {3},
     PAGES = {639--657},
}

@article {BeKa20,
    AUTHOR = {Bekka, B. and Kalantar, M.},
     TITLE = {Quasi-regular representations of discrete groups and
              associated {$C^*$}-algebras},
   JOURNAL = {Trans. Amer. Math. Soc.},
  FJOURNAL = {Transactions of the American Mathematical Society},
    VOLUME = {373},
      YEAR = {2020},
    NUMBER = {3},
     PAGES = {2105--2133},
}

@article {KS2022,
    AUTHOR = {Kalantar, M. and Scarparo, E.},
     TITLE = {Boundary maps, germs and quasi-regular representations},
   JOURNAL = {Adv. Math.},
  FJOURNAL = {Advances in Mathematics},
    VOLUME = {394},
      YEAR = {2022},
     PAGES = {Paper No. 108130, 31},
}

@book{Ross13,
author = {Ross, K. A.},
address = {New York},
booktitle = {Elementary analysis : the theory of calculus},
publisher = {Springer-Verlag},
series = {Undergraduate texts in mathematics},
title = {Elementary analysis: the theory of calculus },
year = {2013},
}

@article {BKKO17,
    AUTHOR = {Breuillard, E. and Kalantar, M. and Kennedy,
              M. and Ozawa, N.},
     TITLE = {{$C^*$}-simplicity and the unique trace property for discrete
              groups},
   JOURNAL = {Publ. Math. Inst. Hautes \'Etudes Sci.},
  FJOURNAL = {Publications Math\'ematiques. Institut de Hautes \'Etudes
              Scientifiques},
    VOLUME = {126},
      YEAR = {2017},
     PAGES = {35--71},
}

@misc{RTV25,
      title={Strict comparison for twisted group {\Cs}-algebras}, 
      author={S. Raum and H. Thiel and E. Vilalta},
      year={2025},
      eprint={2505.18569},
      archivePrefix={arXiv},
      primaryClass={math.OA},
    %note={preprint, \url{https://arxiv.org/abs/2505.18569}},
}

@article {Ro25,
    AUTHOR = {Robert, L.},
     TITLE = {Selfless {\Cs}-algebras},
   JOURNAL = {Adv. Math.},
  FJOURNAL = {Advances in Mathematics},
    VOLUME = {478},
      YEAR = {2025},
     PAGES = {Paper No. 110409, 28},
}

@book {Hi74,
    AUTHOR = {Higman, G.},
     TITLE = {Finitely presented infinite simple groups},
    SERIES = {Notes on Pure Mathematics},
    VOLUME = {No. 8},
 PUBLISHER = {Australian National University, Department of Pure
              Mathematics, Department of Mathematics, I.A.S., Canberra},
      YEAR = {1974},
}

@article {St92,
    AUTHOR = {Stein, M.},
     TITLE = {Groups of piecewise linear homeomorphisms},
   JOURNAL = {Trans. Amer. Math. Soc.},
  FJOURNAL = {Transactions of the American Mathematical Society},
    VOLUME = {332},
      YEAR = {1992},
    NUMBER = {2},
     PAGES = {477--514},
}

@misc{Ta24,
      title={Studying Stein's Groups as Topological Full Groups}, 
      author={O. Tanner},
      year={2024},
      eprint={2312.07375},
      archivePrefix={arXiv},
      primaryClass={math.GR},
}

@book {BiSt,
    AUTHOR = {Bieri, R. and Strebel, R.},
     TITLE = {On groups of {PL}-homeomorphisms of the real line},
    SERIES = {Mathematical Surveys and Monographs},
    VOLUME = {215},
 PUBLISHER = {American Mathematical Society, Providence, RI},
      YEAR = {2016},
     PAGES = {xvii+174},
}

@book{Paulsen2003, place={Cambridge}, series={Cambridge Studies in Advanced Mathematics}, title={Completely Bounded Maps and Operator Algebras}, publisher={Cambridge University Press}, author={Paulsen, V.}, year={2003}, collection={Cambridge Studies in Advanced Mathematics}}

@article {Gl74,
    AUTHOR = {Glasner, S.},
     TITLE = {Topological dynamics and group theory},
   JOURNAL = {Trans. Amer. Math. Soc.},
  FJOURNAL = {Transactions of the American Mathematical Society},
    VOLUME = {187},
      YEAR = {1974},
     PAGES = {327--334},
}

@article {LBMB18,
    AUTHOR = {Le Boudec, A. and Matte Bon, N.},
     TITLE = {Subgroup dynamics and {$C^*$}-simplicity of groups of
              homeomorphisms},
   JOURNAL = {Ann. Sci. \'Ec. Norm. Sup\'er. (4)},
  FJOURNAL = {Annales Scientifiques de l'\'Ecole Normale Sup\'erieure.
              Quatri\`eme S\'erie},
    VOLUME = {51},
      YEAR = {2018},
    NUMBER = {3},
     PAGES = {557--602},
}

@misc{BaFl26,
      title={Selfless inclusions arising from commensurator groups of hyperbolic groups}, 
      author={A. Basu and F. Flores},
      year={2026},
      eprint={2605.12737},
      archivePrefix={arXiv},
      primaryClass={math.GR},
}

@article {Ke20,
    AUTHOR = {Kennedy, M.},
     TITLE = {An intrinsic characterization of {$C^*$}-simplicity},
   JOURNAL = {Ann. Sci. \'Ec. Norm. Sup\'er. (4)},
  FJOURNAL = {Annales Scientifiques de l'\'Ecole Normale Sup\'erieure.
              Quatri\`eme S\'erie},
    VOLUME = {53},
      YEAR = {2020},
    NUMBER = {5},
     PAGES = {1105--1119},
      ISSN = {0012-9593,1873-2151},
}

@article {KiPh00,
    AUTHOR = {Kirchberg, E. and Phillips, N. C.},
     TITLE = {Embedding of exact {$C^*$}-algebras in the {C}untz algebra
              {$\mathscr O_2$}},
   JOURNAL = {J. Reine Angew. Math.},
  FJOURNAL = {Journal f\"ur die Reine und Angewandte Mathematik. [Crelle's
              Journal]},
    VOLUME = {525},
      YEAR = {2000},
     PAGES = {17--53},
}

@misc{GKEPL26,
      title={Selfless reduced amalgamated free products and HNN extensions}, 
      author={D. Gao and S. Kunnawalkam Elayavalli and G. Patchell and L. Teryoshin},
      year={2026},
      eprint={2604.06982},
      archivePrefix={arXiv},
      primaryClass={math.OA},
}

@inproceedings {Ki95,
    AUTHOR = {Kirchberg, E.},
     TITLE = {Exact {${\rm C}^*$}-algebras, tensor products, and the
              classification of purely infinite algebras},
 BOOKTITLE = {Proceedings of the {I}nternational {C}ongress of
              {M}athematicians, {V}ol.\ 1, 2 ({Z}\"urich, 1994)},
     PAGES = {943--954},
 PUBLISHER = {Birkh\"auser, Basel},
      YEAR = {1995},
}

@article {Ph00,
    AUTHOR = {Phillips, N. C.},
     TITLE = {A classification theorem for nuclear purely infinite simple
              {$C^*$}-algebras},
   JOURNAL = {Doc. Math.},
  FJOURNAL = {Documenta Mathematica},
    VOLUME = {5},
      YEAR = {2000},
     PAGES = {49--114},
}

@misc{Ka17,
      title={Uniformly recurrent subgroups and the ideal structure of reduced crossed products}, 
      author={T. Kawabe},
      year={2017},
      eprint={1701.03413},
      archivePrefix={arXiv},
      primaryClass={math.OA},
}

@article {Cl20,
    AUTHOR = {Cleary, S.},
     TITLE = {Regular subdivision in {$\mathbb Z[\frac{1+\sqrt 5}{2}]$}},
   JOURNAL = {Illinois J. Math.},
  FJOURNAL = {Illinois Journal of Mathematics},
    VOLUME = {44},
      YEAR = {2000},
    NUMBER = {3},
     PAGES = {453--464},
}

@article {Cl95,
    AUTHOR = {Cleary, S.},
     TITLE = {Groups of piecewise-linear homeomorphisms with irrational
              slopes},
   JOURNAL = {Rocky Mountain J. Math.},
  FJOURNAL = {The Rocky Mountain Journal of Mathematics},
    VOLUME = {25},
      YEAR = {1995},
    NUMBER = {3},
     PAGES = {935--955},
}

@article {Vi26,
    AUTHOR = {Vigdorovich, I.},
     TITLE = {Selfless reduced {C{$\sp *$}}-algebras of linear groups},
   JOURNAL = {Proc. Lond. Math. Soc. (3)},
  FJOURNAL = {Proceedings of the London Mathematical Society. Third Series},
    VOLUME = {133},
      YEAR = {2026},
    NUMBER = {1},
     PAGES = {Paper No. e70180},
}

@misc{HKER25,
      title={Selfless reduced free product {\Cs}-algebras}, 
      author={B. Hayes and S. Kunnawalkam Elayavalli and L. Robert},
      year={2025},
      eprint={2505.13265},
      archivePrefix={arXiv},
      primaryClass={math.OA},
}

@misc{HKEPR25,
      title={Selfless inclusions of {\Cs}-algebras}, 
      author={B. Hayes and S. Kunnawalkam Elayavalli and G. Patchell and L. Robert},
      year={2025},
      eprint={2510.13398},
      archivePrefix={arXiv},
      primaryClass={math.OA},
}

@misc{FKOCP26,
      title={Pureness and stable rank one for reduced twisted group $\mathrm{C}^\ast$-algebras of certain group extensions}, 
      author={F. Flores and M. Klisse and M. Ó Cobhthaigh and M. Pagliero},
      year={2026},
      eprint={2601.19758},
      archivePrefix={arXiv},
      primaryClass={math.OA},
}

@article{FKOCP25,
    AUTHOR = {Flores, F. and Klisse, M. and \'O{} Cobhthaigh,
              M. and Pagliero, M.},
     TITLE = {Selfless reduced free products and graph products of {$\rm
              C^*$}-algebras},
   JOURNAL = {Adv. Math.},
  FJOURNAL = {Advances in Mathematics},
    VOLUME = {502},
      YEAR = {2026},
     PAGES = {Paper No. 111190},
}

@misc{GJKEPR26,
      title={Selfless C*-correspondences, operator valued C*-probability spaces and completely positive maps}, 
      author={D. Gao and M. Junge and S. Kunnawalkam Elayavalli and G. Patchell and L. Robert},
      year={2026},
      eprint={2607.20361},
      archivePrefix={arXiv},
      primaryClass={math.OA},
}

\bigskip
\bigskip
ADDRESS

\smallskip
\smallskip
Felipe Flores

Department of Mathematics, University of Virginia,

114 Kerchof Hall. 141 Cabell Dr,

Charlottesville, Virginia, United States

E-mail: hmy3tf@virginia.edu

\smallskip
\smallskip
Joseph Gondek

Mathematical Institute, University of Oxford

United Kingdom

E-mail: scro5525@ox.ac.uk

\end{document}